\documentclass[11pt]{amsart}
\usepackage{wrapfig}
\usepackage[english]{babel}
\usepackage{graphicx}
\input{xypic.tex}

\xyoption{all}
\usepackage{amssymb}
\usepackage{amsfonts}
\usepackage{amsmath}
\usepackage{amsthm}

\usepackage[english]{babel}
\usepackage[latin1]{inputenc}
\usepackage{amsmath,amssymb,amsthm,epsfig}
\usepackage{eufrak}
\usepackage{color}
\usepackage{mathrsfs}
\usepackage{latexsym}
\usepackage{times, amscd}

\newtheorem{teo}{Theorem}[section]
\newtheorem{prop}[teo]{Proposition}
\newtheorem{ddef}[teo]{Definition}
\newtheorem{ex}[teo]{Example}
\newtheorem{cor}[teo]{Corollary}

\newtheorem{obs}{Remark}
\newtheorem{lem}[teo]{Lemma}

\newcommand{\msf}{\mathsf}
\newcommand{\scr}{\mathscr}
\newcommand{\Q}{{\mathbb{Q}}}
\newcommand{\F}{{\mathcal{F}}}

\newcommand{\pn}{{\mathbb{P}^n}}
\newcommand{\C}{{\mathbb{C}}}
\newcommand{\Z}{{\mathbb{Z}}}
\newcommand{\wP}{{\mathbb{P}}}
\newcommand{\umapright}[1]
{\hskip3pt\smash{\mathop{\longrightarrow}\limits^{#1}}
\hskip3pt}
\newcommand{\ti}[1]{\tilde{#1}}
\newcommand{\wti}[1]{\widetilde{#1}}

\begin{document}
\title{\textsc{a bott  type residue formula on complex orbifolds}}
\author{ Maur\'icio Corr\^ea Jr., Miguel Rodriguez Peña and Márcio G. Soares}
\address{Maurício  Corrêa  Jr., Dep. Matemática
ICEx - UFMG, Campus Pampulha,
31270-901 Belo Horizonte - Brasil.}
\email{mauricio@mat.ufmg.br}

\address{M. Rodriguez Peña ,
Dep. Matemática
ICEx - UFMG, Campus Pampulha,
31270-901 Belo Horizonte - Brasil}
\email{miguel.rodriguez.mat@gmail.com}

\address{Márcio G. Soares,
Dep. Matemática
ICEx - UFMG, Campus Pampulha
31270-901 Belo Horizonte - Brasil.
}\email{msoares@mat.ufmg.br}

\begin{abstract} We prove  residual formulas for vector fields defined on compact complex orbifolds with isolated singularities and give some applications of these on weighted projective spaces.
\end{abstract}
\maketitle
\begin{scriptsize}
\textit{To Jos\'e Seade on the occasion of his 60th birthday}
\end{scriptsize}

 \footnotetext[3]{{\sl This work was partially supported by CNPq, CAPES, FAPEMIG and FAPERJ.} }

\section*{Introduction}

In this article we prove  residual formulas for vector fields defined on compact complex orbifolds with isolated singularities.  We then derive some applications, related to holomorphic foliations on weighted complex projective spaces and conclude by looking at Hirzebruch surfaces as resolutions of weighted projective planes of a simple type. These residual formulas are inspired by the well known residue formula of R. Bott for vector fields, with isolated zeros, defined on complex manifolds and also by several more recent results on this subject, which can be found in \cite{bss}.

\emph{V-manifolds} were introduced by I. Satake in \cite{sa} and later rediscovered by W. Thurston \cite{thurs}, which referred to them as \emph{orbifolds} and since then they are known by this name. These objects may present singularities, but these are all quotient singularities originating locally from the 
action of finite groups and hence are tractable. More precisely, an orbifold $X$ is a complex space endowed with the following property:  each point $p \in X$ possesses a neighborhood which is the quotient $Y/G_p$, where $Y$ is a complex manifold, say of dimension $n$, and $G_p$ is a properly discontinuous finite group of automorphisms of $Y$, so that locally we have a quotient map  $(Y, \tilde{p}) \umapright{\pi_p} (Y/G_p,p)=(X, p)$. If the isotropy group $G_p$ of $p$ is non-trivial, then $p$ is a singular point of $X$. The structure of orbifolds around its singularities was elucidated by D. Prill in \cite{pr}. We will also make use of results by R. Blache \cite{bla} concerning Chern classes and local Chern classes on orbifolds. An extensive account on orbifolds can be found in \cite{alr}. 

In Section \ref{bottorbi} we state and prove the main results of this note, which we briefly present now. 

\noindent{\bf{Theorem 1}}
\emph{Let $X$ be a compact orbifold of dimension $n$ with only isolated singularities, let $\scr{L}$  be a locally V-free sheaf of rank $1$ over $X$ and $L$ the associated line V-bundle. Suppose ${\xi}$ is a holomorphic section of ${TX\otimes L}$ with isolated zeros. For each zero $p$ of the section $\xi$, let $\tilde{\xi}= \pi_p^\ast \xi$ be a local lifting of $\xi$ via the quotient map $(Y, \tilde{p}) \umapright{\pi_p} (X,p)$ and let $J{\tilde{\xi}}=\displaystyle\left(\frac{\partial\tilde{\xi}_i}{\partial \tilde{z}_j}\right)_{1\leq i, j\leq n}$ be the matrix of the linear part of $\tilde{\xi}$. 
If $P$ is any invariant polynomial of degree $n$ and $L^{\vee}$ is the dual of $L$ then}
$$\int_X P(TX-L^{\vee})=\sum_{p \,|\, \xi(p)=0}\frac{1}{\# G_p}\mbox{Res}_{\,\tilde{p}}\left[\frac{P(J\tilde{\xi} )\,d\tilde{z}_1\wedge\dots\wedge d\tilde{z}_n}{\tilde{\xi}_1\dots\tilde{\xi}_n}\right]$$
\noindent\emph{where ${Res}_{\,\tilde{p}}\displaystyle\left[\frac{P(J\tilde{\xi} )\,d\tilde{z}_1\wedge\dots\wedge d\tilde{z}_n}{\tilde{\xi}_1\dots\tilde{\xi}_n}\right]$ is Grothendieck's point residue.}\\

\noindent{\bf{Theorem 2}}
\emph{Let $X$ be a compact complex orbifold of dimension $n$ with only isolated singularities, $\scr{L}$ a locally V-free sheaf of rank $1$ over $X$ and $L$ the associated line V-bundle. Let $\sigma : \wti{X} \longrightarrow X$ be a good resolution of $X$, $\mathcal{ E}$ the exceptional divisor, and let $\xi$ be a holomorphic section of $TX \otimes L$ with isolated zeros. Assume $(\sigma^\ast \scr{L})^{\vee \vee}$ is locally free. Then}
$$
\sum\limits_{p\,|\, \xi(p)=0} I_p(\xi)= \int\limits_{\wti{X}} c_n \left((\sigma^\ast(TX \otimes L))^{\vee \vee}\right) - \sum\limits_{p \in Sing(X)} \int\limits_{\wti{X}} c_n \left(p, (\sigma^\ast(TX \otimes L))^{\vee \vee}\right).
$$
\emph{In particular, if $\sigma^\ast(\scr{T}_X \otimes \scr{L})$ is reflexive, then}
$$
\sum\limits_{p \in Sing(X) \cap \{ \xi(p)=0\}} I_p(\xi)= \sum\limits_{\ti{p} \in \mathcal{E}\,|\, \sigma^\ast{\xi}(\ti{p})=0} I_{\ti{p}}(\sigma^\ast\xi) - \sum\limits_{p \in Sing(X)} \int\limits_{\wti{X}} c_n \left(p, \sigma^\ast(TX \otimes L)\right),
$$
\noindent\emph{where $I_p(\xi)$ is the Poincaré-Hopf index of the section $\xi$ at $p$}.

In Section \ref{appl}  we exploit some consequences of these results in weighted projective spaces $\wP^n_w$, among which an enumerative result counting the number of singularities of holomorphic foliations on these spaces and also, in a particular but interesting case, a condition for a holomorphic foliation on $\wP(1,1,k)$ to have a singularity at the singular point $\sf{e}_2 =[0:0:1]_k$ of $\wP(1,1,k)$. The latter is shown in two different ways, the first one is based on the enumerative result for $\wP^n_w$ mentioned above, and the second one is obtained by considering the Hirzebruch surfaces $\Sigma_k$ as good resolutions of $\wP(1,1,k)$. It's worth pointing out that if a holomorphic foliation $\mathcal{F}_{|\wP(1,1,k) \setminus \{{\sf{e}_2}\}}$ is given, around the singular point $\sf{e}_2 \in \wP(1,1,k)$, by a vector field $\xi$ or by a differential $1$-form $\eta$, then the lifted vector field $\ti{\xi}$ or the lifted $1$-form $\ti{\eta}$ on $(\C^3, 0)$, via $(\C^3, 0) \umapright{\pi_{\sf{e}_2}} (\C^3/G_{\sf{e}_2},{\sf{e}_2})$, extend holomorphically to $0$ and this extension may be non-vanishing at $0$. If this is the case we say that $\mathcal{F}$ is non-singular at $\sf{e}_2$ (see Example \ref{nonsing}).

\section{Preliminaries}\label{prelim}

\subsection{Orbifolds or V-manifolds}\label{orbi}

We will make use of the theory of Chern classes for pairs $(X, \mathscr{S})$, as developed by R. Blache in \cite{bla}, where $X$ is a complex orbifold with isolated quotient singularities and $\mathscr{S}$ is a coherent sheaf on $X$. In fact, what we will exploit is a little more restricted, namely, the Chern classes for locally free V-sheaves on orbifolds. So, we start by briefly recalling the essential objects we will consider. We refer the reader to \cite{alr},  \cite{bla} and \cite{sa} for the definitions and results on orbifolds that we present in this subsection.

\begin{ddef}\label{v-m}
An orbifold is a connected paracompact complex space $X$ satisfying the property that
each point $x \in X$ has an open neighborhood $U \subset X$ which is a quotient $U \cong V/G$, where $V$ is a complex manifold, $G$ is a finite subgroup of the group of biholomorphisms of $V$, with $V$ and $G$ depending on $x$. $(X,x)$ denotes the germ of the orbifold $X$ at the point $x$ and is referred to as a quotient germ. Morphisms of orbifolds are the holomorphic maps between them.
\end{ddef}

Since the possible singularities appearing in a orbifold $X$ are quotient singularities, we have that $X$ is reduced, normal, Cohen-Macaulay and with only rational singularities. We denote its singular set by $Sing(X)$ and the regular part of $X$ by $X_{reg} = X \setminus Sing(X)$. The dimension of $X$ is the dimension of the complex manifold $X_{reg}$. 

The local structure of orbifolds around its singularities was described by D. Prill in \cite{pr}: if $\dim X =n$, then every quotient germ $(X,x)$ determines a unique (up to conjugation) small finite group $G_x \subset Gl(n, \C)$ such that $(X,x) = (Y,y) / G_x$, where $(Y, y)=(\C^n, 0)$, with natural projection $\pi : (Y,y) \rightarrow (X,x)$ which is called the \textsl{local smoothing covering} of $X$ at $x$. Note that $x \in Sing(X)$ precisely when the isotropy group $G_x$ is non-trivial. We point out that $G_x \cong {\pi}_1 ((X,x) \setminus Sing(X))$ and for this reason $G_x$ is called the \textsl{local fundamental group} of $X$ at $x$. Recall that  $G_x$ is \textit{small} if $\mathrm{codim_{\C}\, Fix}\, G_x \geq 2$, or equivalently, if no element of $G_x$ has $1$ as an eigenvalue of multiplicity $n-1$.

Satake's fundamental idea was to use local smoothing coverings to extend the definition of usual known objects to orbifolds. For instance, smooth differential k-forms on $X$ are $C^\infty$-differential k-forms $\omega$ on $X_{reg}$ such that the pull-back $\pi^\ast \omega$ extends to a $C^\infty$-differential k-form on every local smoothing covering $\pi : (Y,y) \rightarrow (X,x)$ of $X$. 
Hence, if $\omega$ is a smooth $2n$-form on $X$ with compact support $\mathrm{Supp} (\omega) \subset (X,x)$ then, by definition, 
\begin{equation}\label{intx}
\int\limits_X \omega = \frac{1}{\# G_x}\, \int\limits_{Y} \pi^\ast \omega.
\end{equation}
If now the $2n$-form $\omega$ has compact support, we use a partition of unity 
$\{\rho_\alpha, U_\alpha \}_{ \alpha \in {\mathcal A}}$, where $\pi_\alpha: (\widetilde{U}_\alpha, \tilde{p}_\alpha)\longrightarrow (U_\alpha, p_\alpha)$ is a local smoothing covering  and $\sum \rho_\alpha (x) =1$ for all $x \in  \mathrm{Supp} (\omega)$, and set

\begin{equation}\label{intX}
\int\limits_X \omega = \sum\limits_\alpha \int\limits_X \rho_\alpha \omega.
\end{equation}

Analogously to the manifold case we have the sheaf $\mathscr{E}^k$
of $C^\infty$-differential k-forms on $X$, the exterior differentiation $d: \mathscr{E}^k \rightarrow \mathscr{E}^{k+1}$ and $\mathscr{E}^\ast = \bigoplus_{0 \leq k \leq 2n} \mathscr{E}^k$. 

Stokes' formula also holds, in particular, if $X$ has dimension $n$ and $\omega \in \Gamma_c(X, \mathscr{E}^{2n-1})$ (compact supports) then
\begin{equation}\label{stokes}
\int\limits_X d \omega =0.
\end{equation}  
The de Rham complexes $H_{dR}^\ast (X, \C) = \bigoplus\limits_{0 \leq k \leq 2n} H_{dR}^k (X, \C)$ and with compact support $H_{dRc}^\ast (X, \C) = \bigoplus\limits_{0 \leq k \leq 2n} H_{dRc}^k (X, \C)$ are then well defined and Poincaré's duality holds: the pairing
\begin{equation}\label{poindual}
\begin{array}{ccc}
H_{dR}^k (X, \C) \times H_{dRc}^{2n-k} (X, \C) &\longrightarrow& \C\\
([\omega],[\eta]) &\longmapsto& \displaystyle\int\limits_X \omega \wedge \eta
\end{array}
\end{equation}
is nondegenerate. As in the case of manifolds, de Rham's theorem holds on orbifolds and we have $H^\ast(X, \C) \cong H_{dR}^\ast (X, \C)$, where $H^\ast(X, \C)$ is the complex of singular cohomology rings with complex coefficients and, by considering \v{C}ech cohomology we also get $H_{dR}^\ast(X, \C) \cong \check{H}^\ast (X, \C)$.

By invoking Hironaka's resolution of singularities we conclude that, given $x \in X$, there exists an open neighborhood $U \subset X$ of $x$  such that, if $\sigma: \widetilde{U} \longrightarrow U$ is a resolution, $\widetilde{Y}$ is a resolution of the normalization of the fiber product $\widetilde{U} \times_U Y$, where $\pi : Y \longrightarrow U$ is a local smoothing covering, then there is a commutative diagram 
\begin{equation}\label{diag1}
\begin{matrix}
\widetilde{Y} & \umapright{\displaystyle \phi} & Y\\
& & \\
\psi \downarrow& & \downarrow \pi\\
& & \\
\widetilde{U} & \umapright{\displaystyle \sigma} & U\\
\end{matrix}
\end{equation}
with $\phi$ bimeromorphic and biholomorphic over $\pi^{-1} (U \setminus Sing (U))$ and
$\psi$ generically finite of $\deg \psi = \deg \pi$ and unramified over $\widetilde{U} \setminus \sigma^{-1} ( Sing (U))$.

\begin{ddef}\label{g-r}
Given an irreducible normal complex space $Y$ and $W \subsetneqq Y$ a connected compact analytic subset, let $(Y, W)$ denote the germ of $Y$ along $W$. A connected open subset $V \subset Y$ is a good representative of $(Y, W)$ if $W \subset V$, $\partial V$ is a real $C^\infty$-manifold and there is a contraction of $V$ onto $W$. If $V$ is a good representative of $(Y, W)$ we define $H^\ast_{dRc}((Y, W)) := H^\ast_{dRc}(V, \C)$.
\end{ddef}

The definition is good since $H^\ast_{dRc}(V) =H^\ast_{dRc}(V')$ if $V$ and $V'$ are good representatives of $(Y, W)$. Note that if both $Y$ and $W$ are manifolds with $\mathrm{codim}_{\C} W=1$, then $H^k_{dRc}((Y, W)) \cong H^{k-2}_{dR}(W)$ for $k \geq 2$.

\begin{ddef}\label{goodres}
A good resolution of an orbifold $X$ is a resolution $\sigma: \widetilde{X} \longrightarrow X$ whose exceptional divisor $\mathcal{E}$ has only normal crossings. If $X$ has only isolated singularities we also require that $\sigma(\mathcal{E}) = Sing(X)$.
\end{ddef}

Good resolutions always exist.

\subsection{Locally V-free sheaves, V-bundles and Chern classes}\label{vchern}

We recall the concept of a coherent G-sheaf on a complex space:

\begin{ddef}\label{coh}
Let $\mathscr{S}$ be a coherent sheaf on a complex space $X$, let $G$ be a group and $G\longrightarrow Aut(X)$ be a group homomorphism. $G$ acts on $\mathscr{S}$ over $X$ provided:
\begin{itemize}
\item[(i)] given an open subset $U \subseteq X$ such that $g(U)=U$ for all $g \in G$, $G$ acts on the $\C$-linear space $\Gamma(U, \mathscr{S})$
\item[(ii)] these actions are compatible with all restriction maps induced by open inclusions $U' \hookrightarrow U$,
\item[(iii)] $g (f \cdot \xi) = g^\ast f \cdot g(\xi)$ for all $\xi \in \Gamma(U, \mathscr{S})$, $f \in \Gamma(U, \mathscr{O}_U)$, 
$g \in G$.
\end{itemize}
\end{ddef}

In case $G$ acts properly discontinuously on $X$ and $\mathscr{S}$ is a coherent $G$-sheaf on $X$ then, for all $x \in X$ and for all $g \in G$, the action on $X$ induces an identification $\mathscr{O}_{X,x} \simeq \mathscr{O}_{X,g\cdot x}$ and analogously $\mathscr{S}_x \simeq \mathscr{S}_{g\cdot x}$ by the action on $\mathscr{S}$. In particular, if $g \cdot x = x$ for all $g$, then $G$ acts on the stalk $\mathscr{S}_x$.

\begin{ddef}\label{vfree}
Let $(X,x)$ be a quotient germ and $\pi : (Y,y) \longrightarrow (X,x)$ be a local smoothing covering. A coherent sheaf $\mathscr{S}$ on $(X,x)$ is V-free if the following equivalent conditions are satisfied:

\begin{itemize}
\item[(i)] $\mathscr{S}$ is reflexive (that is, $\mathscr{S}$ is isomorphic to $\mathscr{S}^{\vee\vee}$) and $(\pi^\ast \mathscr{S})^{\vee\vee}$ is a free sheaf.

\item[(ii)] There is a free sheaf $\breve{\mathscr{S}}$ on $(Y,y)$ such that $G_x$ acts on $\breve{\mathscr{S}}$ and ${\mathscr{S}} \simeq \pi_{\ast}^{G_x}(\breve{\mathscr{S}})$, where $\pi_{\ast}^{G_x}(\breve{\mathscr{S}})$ is the maximal subsheaf of 
$\pi_{\ast}(\breve{\mathscr{S}})$ on which $G_x$ acts trivially.
\end{itemize}

A coherent sheaf $\mathscr{S}$ on an orbifold $X$ is locally V-free if 
$\mathscr{S}_{\vert (X,x)}$ is V-free for all $x \in X$. 
\end{ddef}

Remark that $\breve{\mathscr{S}}= (\pi^\ast \mathscr{S})^{\vee\vee}$.

\begin{ddef}\label{vbundle}
A V-bundle $E$ on an orbifold $X$ is a holomorphic vector bundle $E_{reg}$ on $X_{reg}$ such that $\mathsf{i}_\ast(\mathscr{O}(E_{reg}))$ is a locally V-free sheaf. Here $\mathsf{i} : X_{reg} \hookrightarrow X$ is the inclusion and  $\mathscr{O}(E_{reg})$ is the sheaf of sections of $E_{reg}$.
\end{ddef}

Observe that a holomorphic vector bundle $E_{reg}$ on $X_{reg}$ is a V-bundle on $X$ if, and only if, for every local smoothing covering $\pi : (Y,y) \longrightarrow (X,x)$, there is a holomorphic vector bundle $\breve{E}_Y$ on $(Y,y)$ and an action of $G_x$ on $\breve{E}_Y$
with $$\left({({\breve{E}}_Y})_{\vert (Y,y) \setminus \mathrm{Fix}\, G_x}\right)/G_x \simeq {E_{reg}}_{\vert (X,x)}.$$ Also, $\breve{E}_Y$ and the $G_x$ action on $\breve{E}_Y$ are uniquely determined by $E_{reg}$. 

Now, locally V-free sheaves have the following property:  if $\mathscr{S}$ is a locally V-free sheaf on $X$, then we can associate to it a V-bundle $E$ on $X$ (unique up to isomorphism) satisfying: $\mathscr{S}$ is isomorphic to the sheaf of sections of $E$.

To construct Chern classes we will use the Chern-Weil procedure via metric connections. First we need to define metrics.

\begin{ddef}\label{vmetric}
Given an orbifold $X$ and a locally V-free sheaf $\mathscr{S}$ on $X$, let $E$ be the V-bundle associated to $\mathscr{S}$. A V-metric $h$ on $\mathscr{S}$ is a hermitian metric on $E_{reg}$ (the restriction of $E$ to $X_{reg}$) such that $\pi^\ast h$ admits an extension to a hermitian metric $\breve{h}$ on $\breve{E}_Y$ for all local smoothing coverings $\pi : (Y,y) \longrightarrow (X,x)$.
\end{ddef}

As in the case of manifolds, by using partitions of unity we can construct V-metrics. 

Given an orbifold $X$ of dimension $n$, a locally V-free sheaf $\scr{S}$ of rank $r$ and a V-metric $h$ on $\scr{S}$ we construct classes $c_i (X, \scr{S}, h) \in H^{2i}_{dR}(X, \C)$, $i=0, \dots, n$, as follows:

On the regular part $X_{reg}$ take the metric connection $\nabla_h$ on $\scr{S}$ associated to $h$ and its curvature operator $K_{\nabla_h} = \nabla_h \circ \nabla_h$, which is represented by a $r \times r$ matrix of smooth 2-forms, say $\Theta_h$, in any local trivialization of $\scr{S}$. If $M(r, \C)$ is the algebra of $r \times r$ matrices over $\C$,  let $C_i : M(r, \C) \longrightarrow \C$ be the $i$-th elementary symmetric invariant polynomial and define the Chern form $c_i(h) := C_i \left(\displaystyle \frac{\sqrt{-1}}{2 \pi}\Theta_h\right) \in \Gamma(X_{reg}, \scr{E}^{2i})$. Doing this on all local smoothing coverings we conclude that in fact $c_i(h) \in \Gamma(X, \scr{E}^{2i})$. It follows, as in the $C^\infty$ case, that $d(c_i(h))=0$ for all $i$ and that
$c_i(\scr{S}) = [c_i(h)] \in H^{2i}_{dR}(x, \C)$ is independent of the metric.

As usual, the \emph{total Chern class} of $\scr{S}$ is defined by $c(\scr{S}):= c_0(\scr{S}) + \dots + c_r(\scr{S})+ \dots + c_n(\scr{S})$, with $c_0(\scr{S}) := 1$. 

If $\scr{T}_X$ is the tangent sheaf of $X$ (which is reflexive), then the \emph{Chern classes} of $X$ are, by definition, $c_i(X) = c_i (\scr{T}_X)$, $i=0, \dots, n$. Remark that $\scr{T}_X$ is locally V-free and ${\scr{T}_X}_{|X_{reg}}$ is the holomorphic tangent sheaf of the complex manifold $X_{reg}$.

The usual properties of the Chern classes hold with the following adjustments: a sequence $\scr{S}_m \longrightarrow \cdots \longrightarrow \scr{S}_1$ of locally V-free sheaves on $X$ is \emph{V-exact} provided the sequence
$$
(\pi^\ast \scr{S}_m)^{\vee \vee} \longrightarrow \cdots \longrightarrow (\pi^\ast \scr{S}_1)^{\vee \vee}
$$
is exact, for all local smoothing coverings of $X$. With this in mind we have:

\begin{obs}\label{cclass}
Properties of the Chern classes.
\begin{itemize}
\item[(i)] if $\scr{S}$ is locally V-free and has rank $r$, then $c_i(\scr{S}) =0$ for $i > r$.

\item[(ii)] $c_i(\scr{S}) = (-1)^i \, c_i(\scr{S}^{\vee})$ for $i \geq 0$.

\item[(iii)] $c(\scr{S})=c(\scr{S}')\, c(\scr{S}'')$ whenever we have a V-exact sequence of locally V-free sheaves
$0 \longrightarrow c(\scr{S}') \longrightarrow c(\scr{S}) \longrightarrow c(\scr{S}'') \longrightarrow 0$.

\item[(iv)] $c_i ((\scr{S} \otimes \scr{L})^{\vee \vee}) = \sum_{k=0}^i {{r-k} \choose {i-k}} c_k(\scr{S})\, c_1(\scr{L})^{i-k}$ for $0 \leq i \leq r$, where $\scr{L}$ is any locally V-free sheaf of rank $1$ on $X$.

\item[(v)] $c(\scr{S}) \in H^\ast (X, \Q)$.
\end{itemize}
\end{obs}
\noindent Property \emph{(v)} above is due to the natural use of local smoothing coverings which implies, in many situations, that we divide by the orders of the groups involved. It follows by showing that 
$$
\xi^r + \xi^{r-1} p^\ast c_1(\scr{S}) + \dots + \xi  p^\ast c_{r-1}(\scr{S}) + p^\ast c_r(\scr{S}) =0
$$
in $H^\ast(\wP(\scr{S}), \Q)$, where $\wP(\scr{S}) \umapright{p} X$  is the projectivization of $\scr{S}$ over $X$, with $\xi \in H^2(\wP(\scr{S}), \Q)$ given by $c_1(\scr{O}_{\wP(\scr{S})} (1))$ and $c_i(\scr{S}) \in H^{2i}(X, \Q)$ uniquely determined.

Chern numbers are defined as in the smooth case, that is, suppose $X$ is a compact orbifold of dimension $n$ and let $\scr{S}$ be a locally V-free sheaf on $X$. For $\mathbb{K}= \Q \: \mathrm{or}\: \C$ the map
\begin{equation}\label{K}
\begin{array}{ccc}
H^{2n}_{dR} (X, \mathbb{K}) &\longrightarrow & \mathbb{K}\\
\, \\
{[ \omega]}  & \longmapsto & \displaystyle [X]  \frown [\omega]
\end{array}
\end{equation}
gives an identification $H_{dR}^{2n} (X, \mathbb{K}) \simeq \mathbb{K}$. In $\mathbb{K}[Z_1, \dots, Z_n]$ give weight (or degree) $i$ to $Z_i$. If $P$ is a quasihomogeneous monomial of degree $n$, $P=  Z_1^{\nu_1}\dots Z_n^{\nu_n}$, $\nu_1 + 2\nu_2 + \dots + n\nu_n =n$, then the number
$P(c(\scr{S})) := c_1(\scr{S})^{\nu_1} \cdots c_n(\scr{S})^{\nu_n} \in \mathbb{K}$, obtained via (\ref{K}), is the \emph{Chern number of $\scr{S}$ with respect to $(\nu_1, \dots, \nu_n)$}. If 
$P \in \mathbb{K}[Z_1, \dots, Z_n]$ is a quasihomogeneous polynomial of degree $n$, then the number $P(c(\scr{S})) \in \mathbb{K}$ is defined in the same way.

\subsection{Local Chern classes}\label{locchern}

R. Blache \cite{bla} developed a natural concept of local Chern classes on orbifolds, which will be very useful for our purposes. We briefly describe it. 

Let $\sigma : (\widetilde{X}, \mathcal{E}) \longrightarrow (X, x)$ be a good resolution of an isolated quotient singularity as in Definition \ref{goodres} and suppose $\breve{\scr{S}}$ is a locally free sheaf on $\widetilde{X}$ such that $\scr{S}=(\sigma_\ast \breve{\scr{S}})^{\vee \vee}$ is a V-free sheaf on $X$. Let $U \subset W$ be open neighborhoods of $x$ with $U$ relatively compact and such that $U$ and $W$ are good representatives of $(X,x)$ as in Definition \ref{g-r}. Then:

\begin{itemize}
\item[(i)] There exist metrics $h$ and $\breve{h}$ on $\scr{S}$ and $\breve{\scr{S}}$, respectively, such that $h_{| W \setminus U} = \breve{h}_{| \widetilde{W} \setminus \widetilde{U}}$.

\item[(ii)] Let $c_i(h) \in \Gamma(W, \scr{E}^{2i})$ and $c_i(\breve{h}) \in \Gamma(\widetilde{W}, \scr{E}^{2i})$ be the Chern forms associated to $(W, \scr{S}, h)$ and
$(\widetilde{W}, \breve{\scr{S}}, \breve{h})$, respectively. Then 
\begin{equation}\label{loccherncl}
c_i (x, \breve{\scr{S}}) := c_i (\widetilde{X}, \mathcal{ E}, \breve{\scr{S}}) := c_i(\breve{h}) - \sigma^\ast c_i(h) \in H_{dRc}^{2i} ((\widetilde{X}, \mathcal{E}), \C )
\end{equation}
does not depend on $U$, $W$, $h$ and $\breve{h}$ and is called the \emph{$i$-th local Chern class of $\breve{\scr{S}}$ along $\mathcal{E}$}.
\end{itemize}

Remark that by $(i)$ the local Chern class $c_i (x, \breve{\scr{S}})$ localizes at $x$. The relation between the local Chern numbers and the Chern numbers is given in Proposition 3.14 of \cite{bla} and reads:

\begin{prop}\label{locglob}

Let $\sigma : \widetilde{X} \longrightarrow X$ be a good resolution of a compact orbifold of dimension $n$ with only isolated singularities and let $\breve{\scr{S}}$ be a locally free sheaf on $\widetilde{X}$ with $\scr{S}=(\sigma_\ast \breve{\scr{S}})^{\vee \vee}$ a locally V-free sheaf on $X$. If $P \in \mathbb{C}[Z_1, \dots, Z_n]$ is an invariant polynomial of degree $n$ then,
\begin{equation}
\int\limits_X P(c(\scr{S})) = \int\limits_{\widetilde{X}} P(c(\breve{\scr{S}})) - \sum\limits_{x \in Sing(X)} \int\limits_{\widetilde{X}} P(c(x, \breve{\scr{S}})).
\end{equation}
\end{prop}

\section{Bott's residue formula on orbifolds}\label{bottorbi}

\subsection{The residue formula}\label{ver1} 

Let $X$ be a complex orbifold of dimension $n$ and $L$ a line 
V-bundle over $X$ and consider the Chern classes
\begin{equation}\label{vir}
 \begin{array}{c}
c_j (TX -L^{\vee}) :=  c_j(X) + c_{j-1}(X) c_1(L) + \cdots + (c_1(L))^j , \quad 1 \leq j \leq n,\\
{}\\
c^\nu(TX-L^{\vee})=c_1^{\nu_1}(TX-L^{\vee})\dots c_n^{\nu_n}(TX-L^{\vee}), \quad \nu=(\nu_1,\dots,\nu_n),\hfill\\
{}\\
n=\nu_1+2\nu_2+\dots+ n\nu_n.\hfill
\end{array}
\end{equation}
Write $C^\nu = C_1^{\nu_1}\dots C_n^{\nu_n}$ for the quasihomogeneous invariant polynomial of degree $n$, where $C_i$ is the i-th elementary symmetric function, and remark that any invariant polynomial $P$ of degree $n$ is a linear combination $P=\sum_\nu a_\nu C^\nu$, $a_\nu \in \C$.

\begin{teo}\label{teo1}
Let $X$ be a compact orbifold of dimension $n$ with only isolated singularities, let $\scr{L}$  be a locally V-free sheaf of rank $1$ over $X$ and $L$ the associated line V-bundle. Suppose ${\xi}$ is a holomorphic section of ${TX\otimes L}$ with isolated zeros. If $P$ is an invariant polynomial of degree $n$, then
$$\int_X P(TX-L^{\vee})=\sum_{p \,|\, \xi(p)=0}\frac{1}{\# G_p}\mbox{Res}_{\,\tilde{p}}\left[\frac{P(J\tilde{\xi} )\,d\tilde{z}_1\wedge\dots\wedge d\tilde{z}_n}{\tilde{\xi}_1\dots\tilde{\xi}_n}\right]$$

\noindent where, for each $p$ such that $\xi(p)=0$, $(\widetilde{U}, \tilde{p}) \umapright{\pi_p} (U,p)$ is a smoothing covering of $X$ at $p$, $\tilde{\xi}= \pi_p^\ast \xi$,  $J{\tilde{\xi}}=\displaystyle\left(\frac{\partial\tilde{\xi}_i}{\partial \tilde{z}_j}\right)_{1\leq i, j\leq n}$ and ${Res}_{\,\tilde{p}}\displaystyle\left[\frac{P(J\tilde{\xi} )\,d\tilde{z}_1\wedge\dots\wedge d\tilde{z}_n}{\tilde{\xi}_1\dots\tilde{\xi}_n}\right]$ is Grothendieck's point residue.

\end{teo}

\begin{proof}
Let $p_1, \dots, p_k$ be the zeros of the section $\xi$ and choose points $p_{k+1}, \dots, p_l$ in $X$ such that:
\begin{itemize}
\item[(i)]  For all $1 \leq \alpha \leq l$,  $\pi_\alpha : (\widetilde{U}_\alpha, \tilde{p}_\alpha) \longrightarrow (U_\alpha, p_\alpha)$ is a smoothing covering, $U_\alpha$ is a good representative of $\{p_\alpha\}$ and $\{U_\alpha\}_{\alpha=1, \dots, l}$ is an open cover of $X$.
\item[(ii)] For each zero $p_\alpha$ of $\xi$, $p_\alpha \not\in U_\beta$ for all $\beta \neq \alpha$. 
\item[(iii)] If $p_\alpha \neq p_\beta$ are zeros of $\xi$ then $U_\alpha \cap U_\beta = \emptyset$, $U_\alpha \cap Sing (X)= \{p_\alpha\}$ in case $p_\alpha \in Sing (X)$ and $U_\alpha \cap Sing (X)= \emptyset$ otherwise.
\item[(iv)] If $p_\alpha \in X_{reg}$ then $U_\alpha$ is a trivializing neighborhood for both $TX$ and $L$, whereas if $p_\alpha \in Sing(X)$, $U_\alpha \setminus \{p_\alpha\}$ is a trivializing neighborhood for both $TX$ and $L$.
\end{itemize}

We follow Chern's arguments \cite{ch} for the first part of the proof. The constructions below are all well defined, see \cite{sa} \textsection 2.  

Equip the locally V-free sheaves $\scr{T}_X$ and $\scr{L}$,  the tangent sheaf of $X$ and the sheaf whose associated V-bundle is $L$, respectively, with hermitian V-metrics $h$ and $H$.  

Locally, using $\pi_p:(\C^n, 0)\cong(\wti{U},\ti{p}) \longrightarrow (U, p)$, we have $\breve{h} = (\breve{h}_{ij})$ and $\breve{H}= a_{\wti{U}} >0$. The (unique) connection which preserves the hermitian structure in $T{\wti{U}}$ has matrix
$$
\omega = \partial \breve{h} \cdot \breve{h}^{-1} ,\; \omega =(\omega_{i j}), \; \omega_{i j} =\sum_k \Gamma^j_{i k} d\ti{z}_k.
$$
The curvature matrix is 
$$
K_{\nabla_h} = d\omega - \omega \wedge \omega = \overline{\partial} \omega =(\Omega_{i j})\;\;\;\mathrm{and} \;\;\; \overline{\partial} K_{\nabla_h} =0.
$$
For $\breve{L}_{\wti{U}}$ the curvature form is
$$
\Phi = -\partial\overline{\partial} \log a_{\wti{U}}, \;\;\; a_{\wti{U}} = a_{\wti{U}'} |\breve{H}_{\wti{U} \wti{U}'}|.
$$
We have
$$
\Phi = - \sum_k d\ti{z}_k \wedge \psi_k, \; \; \psi_k = \overline{\partial}\, \dfrac{\partial \log a_{\wti{U}}}{\partial \ti{z}_k}, \;\;\overline{\partial} \psi_k =0,
$$
Consider the matrix of $C^\infty$ $(1,1)$-forms
$$
\wti{\Omega} = (\wti{\Omega}_{i j}), \; \wti{\Omega}_{i j} = \Omega_{i j} - d \ti{z}_j \wedge \psi_i.
$$

In \cite{ch} it is shown that, for any invariant polynomial of degree $n$, whenever $\wti{U}_\alpha \cap \wti{U}_\beta \neq \emptyset$ the matrices $\wti{\Omega}_\alpha$ satisfy
 $P( \wti{\Omega}_\alpha)= P(\wti{\Omega}_\beta)$
and hence they define \emph{a global smooth $(n,n)$-form $P(\wti{\Omega})$ on $X$}. 

Outside the zeros
of the section $\xi$, the form $P(\wti{\Omega})$ satisfies 
\begin{equation}\label{loc1}
 P( \wti{\Omega} ) = d \Psi
\end{equation}
{where} $\Psi$ {is a smooth 
$(n, n-1)$-form on} $X \setminus \{p : \xi(p)=0\}$ and

\begin{equation}\label{loc2}
\left[\left(\dfrac{\sqrt{-1}}{2\pi}\right)^n P(\wti{\Omega})\right]= 
P(TX -L^{\vee}) \in H^{2n}_{dR} (X, \C).
\end{equation}

Write
\begin{equation}\label{qsi}
\ti{\xi}_\alpha = \sum\limits_i \ti{\xi}_i^\alpha \dfrac{\partial}{\partial \ti{z}_i^\alpha}, \qquad \ti{\xi}_i^\alpha \in \Gamma(\wti{U}_\alpha, \breve{L}_{|\wti{U}_\alpha})
\end{equation}
for the local expression of $\ti{\xi}$ over $\wti{U}_\alpha$.

At each $p_\alpha$ which is a zero of $\xi$, adopt the standard hermitian metrics ${h}_{ij}^\alpha = \delta_{ij}$ and ${H}^\alpha =1$ over ${U}_\alpha$ and extend these to $h$ and $H$, in $\scr{T}_X$ and $\scr{L}$ respectively, by means of a partition of unity. Then, with the induced metrics over $\ti{U}_\alpha$,  the $(n,n-1)$-form $\Psi$ above has the explicit expression (equation (55) of \cite{ch}):

\begin{equation}\label{Psi}
\pi^\ast_\alpha\Psi = (-1)^{n+1} P(J\ti{\xi}_\alpha)\; \eta_\alpha \wedge \;(\overline{\partial} \,\eta_\alpha)^{n-1}
\end{equation}
where
$$
\eta_\alpha = \dfrac{\sum\limits_i \overline{\ti{\xi}_i^\alpha} \;d \ti{z}_i^\alpha}{\sum\limits_i \overline{\ti{\xi}_i^\alpha}\; \ti{\xi}_i^\alpha}
$$
is the $(0,1)$-form dual to $\ti{\xi}_\alpha$.

We now investigate $\eta_\alpha \wedge \;(\overline{\partial}\eta_\alpha)^{n-1}$. We have
\begin{equation}\label{Psi1}
\eta_\alpha \wedge (\overline{\partial}\eta_\alpha)^{n-1} \!\!=
\!\dfrac{{\upsilon}_n}{|\ti{\xi}_\alpha|^{2n}} \!\!\left[d\ti{z}_1^\alpha \wedge \dots \wedge d\ti{z}_n^\alpha \!\bigwedge\!
\sum\limits_i (-1)^{i-1}  \overline{\ti{\xi}_i^\alpha} \, d\,\overline{\ti{\xi}_1^\alpha} \wedge \dots \wedge \widehat{d\,\overline{\ti{\xi}_i^\alpha}} \wedge \dots \wedge d\,\overline{\ti{\xi}_n^\alpha} \right]\!
\end{equation}
where ${\upsilon}_n = (-1)^{\frac{n(n-1)}{2}}(n-1)!$.

Recall that the Bochner-Martinelli kernel in $\C^n \times \C^n$ has the representation

$$
\mathsf{K}(w,u) = \frac{(n-1)!}{(2\pi \sqrt{-1})^n |w-u|^{2n}}
\,\sum\limits_i\,(\overline{w_i} - \overline{u_i})\,dw_i \,\wedge
\left[ \bigwedge\limits_{j \neq i} d\,\overline{w_j} \wedge dw_j \right].
$$
A manipulation shows that
$$
\mathsf{K}(w,0)=  (-1)^{\frac{n(n-1)}{2}}\frac{(n-1)!}{(2\pi \sqrt{-1})^n |w|^{2n}} \sum_i \Theta_i(w) \wedge \vartheta (w)
$$
where $\Theta_i(w) ={(-1)}^{i-1}\overline{w_i}\, d\,\overline{w_1}\wedge \cdots \wedge \widehat{d\,\overline{w_i}} \wedge \cdots \wedge d\,\overline{w_n} $ and $\vartheta(w) = dw_1 \wedge \cdots \wedge dw_n$.
Set $\ell_n = {(-1)}^{\frac{n(n-1)}{2}}\,{(n-1)! \over {(2 \pi \sqrt{-1})}^n}$ and put
$$
\mathsf{B} (z, \zeta) = \frac{\ell_n}{|z - \zeta|^2} \, \sum\limits_i \,\Theta_i (z - \zeta)\wedge \vartheta (\zeta).
$$
This is the same as
$$
\mathsf{B}(z, \zeta) =\frac{\ell_n}{|z - \zeta|^2} \, \sum\limits_{i} {(-1)}^{i-1} (\overline{z_i} - \overline{\zeta_i}) \bigwedge
\limits_{j \neq i} (d\overline{z_j} - d\overline{\zeta_j}) \wedge d\zeta_1 \wedge \cdots \wedge d\zeta_n.
$$
Now, it is proven in \cite{G-H}, Lemma in page 651, that
\begin{equation}\label{resdol}
\mbox{Res}_{\,0}\left[\frac{g \,dz_1\wedge\dots\wedge dz_n}{f_1\dots f_n}\right] = \int\limits_{S^{2n-1}_\epsilon} g\, \Phi^\ast \mathsf{B}
\end{equation}
where $g$ and $f=( f_1, \dots, f_n)$  are holomorphic in a neigborhood of $0 \in \C^n$, $f$ has an isolated zero at $0$, $\Phi : (\C^n, 0) \longrightarrow \C^n \times \C^n $  is the map $\Phi(z) = (z+f(z), z)$ and  $S^{2n-1}_\epsilon$ is the euclidean sphere of radius $\epsilon$ centered at $0$, for all sufficiently small $\epsilon>0$. 

By replacing $f$ by $\xi_\alpha$ and considering the map $\Phi : (\wti{U}_\alpha, \ti{p}_\alpha) \longrightarrow \C^n \times \C^n$, $\Phi(\ti{z}_\alpha) = (\ti{z}_\alpha + \ti{\xi}_\alpha(\ti{z}_\alpha), \ti{z}_\alpha)$, a straightforward calculation gives
\begin{equation}
\Phi^\ast \mathsf{B} =\left(\dfrac{1}{2\pi \sqrt{-1}}\right)^n  \eta_\alpha \wedge \;(\overline{\partial}\eta_\alpha)^{n-1}.
\end{equation}
By (\ref{loc2}) it follows that
\begin{equation}\label{end1}
\int\limits_X P(TX - L^\vee) = \int\limits_X \left[ \left( \frac{\sqrt{-1}}{2\pi}\right)^n P(\wti{\Omega})\right].
\end{equation}
For each zero $p_\alpha$ of $\xi_\alpha$ choose a euclidean ball $B_\epsilon(p_\alpha)$ centered at $p_\alpha$ such that $\overline{B_\epsilon(p_\alpha)} \subset U_\alpha$.  By Stokes and (\ref{loc1})
\begin{equation}\label{c1}
\displaystyle \int\limits_{X \setminus \cup_1^k B_\epsilon(p_\alpha)} P \left( \frac{\sqrt{-1}}{2\pi} \wti{\Omega}\right) = 
\int\limits_{X \setminus \cup_1^k B_\epsilon(p_\alpha)} \left(\dfrac{\sqrt{-1}}{2\pi}\right)^n \; d\Psi = 
\sum\limits_1^k \int\limits_{\partial B_\epsilon(p_\alpha)} - \left(\dfrac{\sqrt{-1}}{2\pi}\right)^n \Psi.
\end{equation}
Now,
\begin{equation}\label{c2}
\int\limits_{S^{2n-1}_\epsilon(p_\alpha)} - \left(\dfrac{\sqrt{-1}}{2\pi}\right)^n \Psi = \dfrac{1}{\# G_{p_\alpha}}\int\limits_{S^{2n-1}_\epsilon(\ti{p}_\alpha)} - \left(\dfrac{\sqrt{-1}}{2\pi}\right)^n \pi^\ast_\alpha\Psi 
\end{equation}
\begin{equation}\label{c3}
 = \dfrac{1}{\# G_{p_\alpha}}\int\limits_{S^{2n-1}_\epsilon(\ti{p}_\alpha)} - \left(\dfrac{\sqrt{-1}}{2\pi}\right)^n (-1)^{n+1} P(J\ti{\xi}_\alpha)\; \eta_\alpha \wedge \;(\overline{\partial}\eta_\alpha)^{n-1}
\end{equation}
\begin{equation}\label{c4}
= \dfrac{1}{\# G_{p_\alpha}}\int\limits_{S^{2n-1}_\epsilon(\ti{p}_\alpha)} - \left(\dfrac{\sqrt{-1}}{2\pi}\right)^n (-1)^{n+1} P(J\ti{\xi}_\alpha)\; (2\pi \sqrt{-1})^n \Phi^\ast \mathsf{B} 
\end{equation}
\begin{equation}\label{c5}
= \dfrac{1}{\# G_{p_\alpha}}\int\limits_{S^{2n-1}_\epsilon(\ti{p}_\alpha)} P(J\ti{\xi}_\alpha)\; \Phi^\ast \mathsf{B} = \frac{1}{\# G_{p_\alpha}}\mbox{Res}_{\,\tilde{p}_\alpha}\left[\frac{P(J\tilde{\xi}_\alpha )\,d\tilde{z}_1^\alpha\wedge\dots\wedge d\tilde{z}^\alpha_n}{\tilde{\xi}^\alpha_1\dots\tilde{\xi}^\alpha_n}\right]
\end{equation}
by (\ref{resdol}). Hence
\begin{equation}\label{c6}
\displaystyle \int\limits_{X \setminus \cup_1^k B_\epsilon(p_\alpha)} P \left( \frac{\sqrt{-1}}{2\pi} \wti{\Omega}\right) = \sum_{p_\alpha \,|\, \xi(p_\alpha)=0}\frac{1}{\# G_{p_\alpha}}\mbox{Res}_{\,\tilde{p}_\alpha}\left[\frac{P(J\tilde{\xi}_\alpha )\,d\tilde{z}_1^\alpha\wedge\dots\wedge d\tilde{z}_n^\alpha}{\tilde{\xi}_1^\alpha\dots\tilde{\xi}_n^\alpha}\right].
\end{equation}

Taking $\epsilon \rightarrow 0$ we get, by (\ref{end1}), 
\begin{equation}\label{end2}
\int\limits_X P(TX - L^\vee) = \sum_{p \,|\, \xi(p)=0}
\frac{1}{\# G_p}\mbox{Res}_{\,\tilde{p}}\left[\frac{P(J\tilde{\xi} )\,d\tilde{z}_1\wedge\dots\wedge d\tilde{z}_n}{\tilde{\xi}_1\dots\tilde{\xi}_n}\right]
\end{equation}
as stated.
\end{proof}

\subsection{Another residue formula}\label{ver2} 

In this subsection we bring in resolutions and obtain relations between characteristic numbers of the orbifold and of a good resolution. As in \textbf{\ref{ver1}}, let $X$ be a compact complex orbifold of dimension $n$ with only isolated singularities, $\scr{L}$ a locally V-free sheaf of rank $1$ over $X$ and $L$ the associated line V-bundle. 

Suppose $\xi$ is a holomorphic section of $TX \otimes L$ with isolated zeros. The Poincar\' e-Hopf index of $\xi$ at a zero $p$ was defined by Satake in \textsection 3.2 of \cite{sa} and is: 
\begin{equation}\label{ph1}
I_p (\xi) = \frac{1}{\#G_p} I_{\ti{p}} (\ti{\xi})
\end{equation}
where $\pi_p : (\wti{U}, \ti{p}) \longrightarrow (U, p)$ is a local smoothing covering, $\ti{\xi} = \pi_p^\ast \xi$ and $I_{\ti{p}} (\ti{\xi})$ is the usual Poincar\' e-Hopf index. It's well known that, in terms of point residues, using $\det$ to designate the $n$-th elementary symmetric polynomial $C_n$,
\begin{equation}\label{ph2}
I_{\ti{p}} (\ti{\xi}) = \mbox{Res}_{\,\tilde{p}}\left[\frac{\det(J\tilde{\xi} )\,d\tilde{z}_1\wedge\dots\wedge d\tilde{z}_n}{\tilde{\xi}_1\dots\tilde{\xi}_n}\right].
\end{equation}

Recall from (\ref{vir}) and Remark \ref{cclass} (iv) that
\begin{equation}\label{virtop}
c_n(TX \otimes L) = c_n (TX - L^\vee) =  c_n(X) + c_{n-1}(X) c_1(L) + \cdots + (c_1(L))^n.
\end{equation}

\begin{teo}\label{teo2}
Let $X$ be a compact complex orbifold of dimension $n$ with only isolated singularities, $\scr{L}$ a locally V-free sheaf of rank $1$ over $X$ and $L$ the associated line V-bundle. Let $\sigma : \wti{X} \longrightarrow X$ be a good resolution of $X$, $\mathcal{ E}$ the exceptional divisor, and let $\xi$ be a holomorphic section of $TX \otimes L$ with isolated zeros. Assume $(\sigma^\ast \scr{L})^{\vee \vee}$ is locally free. Then
\begin{equation}\label{teo2a}
\sum\limits_{p\,|\, \xi(p)=0} I_p(\xi)= \int\limits_{\wti{X}} c_n \left((\sigma^\ast(TX \otimes L))^{\vee \vee}\right) - \sum\limits_{p \in Sing(X)} \int\limits_{\wti{X}} c_n \left(p, (\sigma^\ast(TX \otimes L))^{\vee \vee}\right).
\end{equation}
In particular, if $\sigma^\ast(\scr{T}_X \otimes \scr{L})$ is reflexive, then
\begin{equation}\label{teo2b}
\sum\limits_{p \in Sing(X) \cap \{ \xi(p)=0\}} I_p(\xi)= \sum\limits_{\ti{p} \in \mathcal{E}\,|\, \sigma^\ast{\xi}(\ti{p})=0} I_{\ti{p}}(\sigma^\ast\xi) - \sum\limits_{p \in Sing(X)} \int\limits_{\wti{X}} c_n \left(p, \sigma^\ast(TX \otimes L)\right).
\end{equation}
\end{teo}
\begin{proof}
First we remark that since $\scr{L}$ is locally V-free, we may construct a resolution such that $(\sigma^\ast \scr{L})^{\vee \vee}$ is locally free (Lemma 3.7 of \cite{bla}). By Proposition  \ref{locglob}, Theorem \ref{teo1}, (\ref{ph1}) and (\ref{ph2}) we have
(\ref{teo2a}). (\ref{teo2b}) follows from (\ref{teo2a}) since $\sigma({\mathcal{E}}) = Sing(X)$ and $I_p(\xi) =I_{\ti{p}}(\sigma^\ast\xi)$ for all $p \not\in Sing(X)$. 
\end{proof}

\section{Applications}\label{appl}

\subsection{One-dimensional foliations on $\wP^n_w$}

Here we consider weighted complex projective spaces with only isolated singularities, which we briefly recall.

Let $w_{0},...,w_{n}$ be positive integers two by two co-primes, set $w:=(w_{0},...,w_{n})$ and $|w|:= w_0+ \dots + w_n$. Define an action of $\C^\ast$ in $\C^{n+1} \setminus \{0\}$ by

\begin{equation}
\begin{array}{ccc}
\C^\ast \times \C^{n+1} \setminus \{0\} & \longrightarrow & \C^{n+1} \setminus \{0\} \\
\lambda . (z_0, \dots, z_n) & \longmapsto &  (\lambda^{w_0} z_0, \dots, \lambda^{w_n} z_n)\\
\end{array}
\end{equation}
and let $\wP_w^n := \C^{n+1} \setminus \{0\} / \sim$. The weights are chosen to be 2 by 2 co-primes in order to assure a finite number of singularities and to give $\wP_w^n$ the structure of an effective, abelian, compact orbifold of dimension $n$.  The singular locus is:
\begin{equation}\label{spnw}
Sing (\wP_w^n) = \{\mathsf{e}_i = [0: \dots:\underbrace{1}_i : \dots : 0]_w \,: i=0,1, \dots, n\}.
\end{equation} 

We have the canonical projection  
\begin{equation}
\begin{array}{ccc}
\pi: {\C}^{n+1} \setminus \{0\} & \longrightarrow& \wP_w^n \hfill \\
(z_0, \dots, z_n) & \longmapsto & [z_0^{w_0}: \dots : z_n^{w_n}]_w
\end{array}
\end{equation}
and the natural map
\begin{equation}
\begin{array}{ccc}
\hskip 32pt \varphi_w: \wP^n & \longrightarrow& \wP_w^n \hfill \\
{[z_0: \dots : z_n]}  & \longmapsto &  [z_0^{w_0}: \dots : z_n^{w_n} ]_w
\end{array}
\end{equation}
of degree $\deg \varphi_w= w_0 \dots w_n$. The map $\varphi_w$ is \emph{good} in the sense of \cite{cr}  \textsection{4.4}, which means, among other things, that V-bundles behave well under pullback. 
It is shown in \cite{mann} that there is a line V-bundle $\scr{O}_{\wP^n_w} (1)$ on $\wP^n_w$, unique up to isomorphism, such that
\begin{equation}
\varphi_w^\ast \scr{O}_{\wP_w^n} (1) \cong  \scr{O}_{\wP^n} (1)
\end{equation}
and, by naturality, $c_1(\varphi_w^\ast \scr{O}_{\wP_w^n} (1)) = c_1 ( \scr{O}_{\wP^n} (1))= \varphi_w^\ast c_1(\scr{O}_{\wP_w^n} (1))$,  from which we obtain the Chern number
\begin{equation}\label{c1^n}
[\wP_w^n] \frown \left( c_1( \scr{O}_{\wP_w^n} (1))\right)^n = \int\limits_{\wP_w^n} \left( c_1( \scr{O}_{\wP_w^n} (1))\right)^n 
= \dfrac{1}{w_0 \dots w_n}
\end{equation}
since
\begin{equation}
1 = \int\limits_{\wP^n} \left(c_1(\scr{O}_{\wP^n} (1))\right)^n = \int\limits_{\wP^n} \varphi_w^\ast\left( c_1( \scr{O}_{\wP_w^n} (1))\right)^n = (\deg \varphi_w) \int\limits_{\wP_w^n} 
\left( c_1( \scr{O}_{\wP_w^n} (1))\right)^n.
\end{equation}
As usual we set $\scr{O}_{\wP^n_w}(k) := \scr{O}_{\wP^n_w}(1)^{\otimes k}$ for $k \in \Z$. 

The Euler sequence on $\wP^n_w$ reads

\begin{equation}\label{euler}
0\longrightarrow \underline{\mathbb{C}} \longrightarrow \bigoplus_{i=0}^{n} \scr{O}_{\wP_w^n}(w_i)\longrightarrow T\mathbb{P}_{w}^{n}\longrightarrow 0
\end{equation}
where
\begin{itemize}
  \item[(i)] $1\longmapsto ({w_0}{z_0},...,{w_n}{z_n})$.
  \item[(ii)] $(P_{0},...,P_{n})\longmapsto \pi_{\ast}\left(\sum_{i=0}^{n}P_{i}\frac{\partial}{\partial z_{i}}\right)$.
\end{itemize}

Twisting by $\scr{O}_{\wP^n_w}(d-1)$ we get
\begin{equation}\label{eulertwis}
0 \longrightarrow \scr{O}_{\wP^n_w}(d-1) \longrightarrow \bigoplus_{i=0}^{n} \scr{O}_{\wP_w^n}(w_i +d-1)\longrightarrow T\mathbb{P}_{w}^{n} \otimes \scr{O}_{\wP^n_w}(d-1) \longrightarrow 0.
\end{equation}
\begin{lem} \label{citaa}
If $d>1-\max\limits_{i \neq j} \{w_{i}+w_{j}\}$, then 
$H^{0}(\mathbb{P}_{w}^{n},T\mathbb{P}_{w}^{n}\otimes\mathcal{O}_{w}(d-1))$ $\neq 0$.
\end{lem}
\begin{proof}
\noindent The pairing $\Omega^{r}_{\mathbb{P}_{w}^{n}}\times \Omega^{n-r}_{\mathbb{P}_{w}^{n}}\longrightarrow \Omega^{n}_{\mathbb{P}_{w}^{n}}$
induces $\Omega^{r}_{\mathbb{P}_{w}^{n}} \cong \mathrm{Hom}_{\mathbb{C}}\left(\Omega^{n-r}_{\mathbb{P}_{w}^{n}}, \Omega^{n}_{\mathbb{P}_{w}^{n}}\right)$.
On the other hand, $K_{\mathbb{P}^{n}_{w}}=\scr{O}_{\wP^n_w}\left(-|w|\right)$ and hence $T\mathbb{P}^{n}_{w}=
(\Omega^{1}_{\mathbb{P}^{n}_{w}})^{\vee}=
\Omega^{n-1}_{\mathbb{P}^{n}_{w}}\otimes
(\Omega^{n}_{\mathbb{P}^{n}_{w}})^{\vee}=
\Omega^{n-1}_{\mathbb{P}^{n}_{w}}(|w|)$. Since $H^{0}(\mathbb{P}_{w}^{n},T\mathbb{P}_{w}^{n}\otimes
\scr{O}_{\wP^n_w}(d-1))\cong
H^{0}(\mathbb{P}_{w}^{n}, 
\Omega^{n-1}_{\mathbb{P}_{w}^{n}}(|w|+d-1))$, the result follows from Corollary 2.3.4 of \cite{dolga}.
\end{proof}

Hence, if $\msf{X}$ is a quasi-homogeneous vector field of type $w$ and degree $d$ in $\C^{n+1}$, as in Lemma \ref{citaa}, that is, writing $\msf{X} = \sum\limits_{i=0}^n P_i (z) \frac{\partial}{\partial z_i}$ we have $P_i (\lambda^{w_0}z_0, \dots, \lambda^{w_n}z_n) = \lambda^{d + w_i - 1}P_i (z_0, \dots, z_n)$, then $\sf{X}$ defines a one-dimensional holomorphic foliation $\F$ on $\wP^n_w$ and, due to (\ref{eulertwis}), $g\msf{R}_w + \msf{X}$ defines the same foliation, where ${\msf{R}}_w = \sum_0^n w_i z_i  \frac{\partial}{\partial z_i}$ is the weighted radial vector field and $g$ is a quasi-homogeneous polynomial of type $w$ and degree $d$. The integer $d$ is, by definition, the \emph{degree} of the foliation $\F$.

Dually, as $H^{0}(\mathbb{P}_{w}^{n},T\mathbb{P}_{w}^{n}\otimes
\scr{O}_{\wP^n_w}(d-1))\cong
H^{0}(\mathbb{P}_{w}^{n}, 
\Omega^{n-1}_{\mathbb{P}_{w}^{n}}(|w|+d-1))$ we have that a one-dimensional holomorphic foliation on $\wP^n_w$ is induced by a quasi-homogeneous $(n-1)$-form $\eta =  \sum\limits_0^n A_i dz_1 \wedge \dots \wedge\widehat{d z_i}\wedge \dots \wedge dz_n$, where $A_i$ is quasi-homoegenous of degree $d +|w| -w_i -1$, subjected to the condition $\sum\limits_0^n w_i z_i A_i \equiv 0$.

We now give an enumerative result concerning foliations on $\wP^n_w$ and examine some consequences of it.

\begin{prop}\label{enum1}
Let $\xi$ be a holomorphic section of $T{\wP}_{w}^{n}\otimes\scr{O}_{\wP^n_w}(d-1)$ with only isolated zeros. Then
\begin{equation}\label{enum2}
 \sum_{p | \xi(p)=0}\mathcal{I}_{p}(\xi)=\frac{1}{w_{0} \dots w_{n}}\sum_{j=0}^{n} C_j(w) (d-1)^{n-j},
\end{equation}
where $C_j(w)$ is the $j$-th elementary symmetric function of the variables $w_0, \dots, w_n$.
\end{prop}

\begin{proof}
By Theorem \ref{teo1}, (\ref{ph1}), (\ref{ph2}) and (\ref{virtop}) we have that
\begin{equation}\label{coeff1}
\sum_{p | \xi(p)=0}\mathcal{I}_{p}(\xi)= \int_{\wP^n_w} c_n(T{\wP}_{w}^{n}\otimes\scr{O}_{\wP^n_w}(d-1)).
\end{equation}
To calculate this Chern number we use the map $\varphi_w$. By (\ref{virtop}),
\begin{equation}\label{coeff2}
\begin{array}{c}
c_n(T{\wP}_{w}^{n}\otimes\scr{O}_{\wP^n_w}(d-1))= \hfill \\
{}\\
= c_n(T{\wP}_{w}^{n}) + c_{n-1}(T{\wP}_{w}^{n}) c_1(\scr{O}_{\wP^n_w}(d-1)) + \cdots + (c_1(\scr{O}_{\wP^n_w}(d-1)))^n.
\end{array}
\end{equation}
By (\ref{euler}) and Remark \ref{cclass} $(iii)$,
\begin{equation}\label{euler1}
c(T{\wP}_{w}^{n}) = c \left( \bigoplus_{i=0}^{n} \scr{O}_{\wP_w^n}(w_i)\right).
\end{equation}
By naturality
\begin{equation}
\varphi^\ast_w \left( c \left( \bigoplus_{i=0}^{n} \scr{O}_{\wP_w^n}(w_i)\right) \right) = c \left( \bigoplus_{i=0}^{n} \varphi^\ast \scr{O}_{\wP_w^n}(w_i)\right) = c \left( \bigoplus_{i=0}^{n} \scr{O}_{\pn}(w_i)\right)
\end{equation}
and
\begin{equation}\label{coeff3}
c \left( \bigoplus_{i=0}^{n} \scr{O}_{\pn}(w_i)\right) = \prod\limits_{i=0}^n \left( 1 + w_i \mathsf{h} \right) = 1+  \sum\limits_{i=1}^n C_i(w) \mathsf{h}^i
\end{equation}
where $\mathsf{h}$ is the hyperplane class in $\pn$. Hence, from (\ref{coeff2}) we get
\begin{equation}
\begin{array}{c}
\varphi^\ast_w c_n(T{\wP}_{w}^{n}\otimes\scr{O}_{\wP^n_w}(d-1))= \hfill  \\
{}\\
=\left[ C_n(w) + C_{n-1}(w) (d-1) + \dots +
C_1(w) (d-1)^{n-1} + (d-1)^n \right] \mathsf{h}^n.
\end{array}
\end{equation}
Integration gives:
\begin{equation}
\begin{array}{ccc}
\displaystyle \sum_{i=0}^n C_i(w) (d-1)^{n-i} &=& \displaystyle\int\limits_{\pn} \varphi_w^\ast c_n(T{\wP}_{w}^{n}\otimes\scr{O}_{\wP^n_w}(d-1))\\
&=&( \deg \varphi_w)\displaystyle \int\limits_{\wP_w^n} c_n(T{\wP}_{w}^{n}\otimes\scr{O}_{\wP^n_w}(d-1))
\end{array}
\end{equation}
and the result follows.
\end{proof}

\begin{ddef}\label{nondeg}
Let $\xi$ be a holomorphic section of $T{\wP}_{w}^{n}\otimes\scr{O}_{\wP^n_w}(d-1)$ and let $p$ be an isolated zero of $\xi$. $p$ is non-degenerate if, given a local smoothing covering $\pi : (\widetilde{U}, \tilde{p}) \longrightarrow (U, p)$, the lifted vector field $\tilde{\xi}= \pi^\ast \xi$ has a non-degenerate singularity at $\tilde{p}$, that is, $\tilde{I}_{\tilde{p}}(\tilde{\xi}) =1$.
\end{ddef}
Given $\mathsf{e}_j \in Sing(\wP_w^n)$ (recall (\ref{spnw})), the local fundamental group $G_{\mathsf{e}_j}$ of $\wP^n_w$ at $\mathsf{e}_j$ is isomorphic to $\boldsymbol{\zeta}_{w_j}$, the group of the $w_j$-th roots of unity. As an immediate consequence of Proposition \ref{enum1} we have

\begin{cor}\label{enum3}
Let $\xi$ be a holomorphic section of $T{\wP}_{w}^{n}\otimes\scr{O}_{\wP^n_w}(d-1)$ with only isolated and non-degenerate zeros. Then
\begin{equation}\label{numnondeg}
\sum\limits_{p | \xi(p)=0} \frac{w_0 \dots w_n}{\#G_p} = \sum_{j=0}^{n} C_j(w) (d-1)^{n-j}. 
\end{equation}
If $\F$ is the foliation induced by $\xi$ and $Sing(\F) \cap Sing(\wP^n_w) = \emptyset$, then (\ref{numnondeg}) reads
\begin{equation}\label{numnondeg1}
(w_0 \dots w_n) \,  \# Sing(\F) = \sum_{j=0}^{n} 
C_j(w) (d-1)^{n-j}.
\end{equation}
\end{cor}

\begin{cor}\label{enum4}
Let $\xi$ be a holomorphic section of $T{\wP}_{w}^{n}\otimes\scr{O}_{\wP^n_w}(d-1)$ with only isolated zeros and denote by $\F$ the induced foliation. Then,
\begin{itemize}
\item[(i)] $Sing(\F) \neq \emptyset$ whenever $d \geq 1$ and, for $n=2$, whenever $d \geq 0$.
\item[(ii)] $Sing(\F) \neq \emptyset$ in case $d-1\, \nmid \, C_n(w)$.
\end{itemize}
\end{cor}
\begin{proof}
(i) holds since the right hand side of (\ref{enum2}) is positive for $d \geq 1$. The case $n=2$ is also straightforward. As for (ii), recall Lemma \ref{citaa} and notice that, if $Sing(\F) = \emptyset$ then (\ref{enum2}) becomes $\sum_{j=0}^{n-1} C_j(w) (d-1)^{n-j} = - C_n (w)$ and, by $(i)$, necessarily $d \leq 0$.
\end{proof}

\begin{obs}
\rm{Let $\F$ be a foliation of degree $d \neq 1$ on $\wP^2_w$, $w=(w_0, w_1, w_2)$, with isolated singularities only. If all points in $Sing(\F)$ are non-degenerate singularities, then $Sing(\F) \neq \{\mathsf{e}_0, \mathsf{e}_1, \mathsf{e}_2\}$. To see this assume $Sing(\F) = \{\mathsf{e}_0, \mathsf{e}_1, \mathsf{e}_2\}$. Then (\ref{enum2}) reads  $\dfrac{C_2(w)}{w_0 w_1 w_2}=
\dfrac{(d-1)^2+(d-1)C_{1}(w)+C_{2}(w)}{w_{0}w_{1}w_{2}}$. Since  $d \neq 1$ we are left with $d=1-C_{1}(w)$. But this contradicts Lemma \ref{citaa} since $d>1-\max\limits_{i \neq j} \{w_{i}+w_{j}\}>
1-C_{1}(w)$.}
\end{obs}

For a complex surface $\ti{X}$ and $\tilde{\F}$ a holomorphic foliation on $\ti{X}$ with isolated singularities, the Baum-Bott index of $\tilde{\F}$ at $\ti{p} \in Sing(\ti{\F})$ is defined by
$$
B\!B_{\ti{p}}(\ti{\F}) = \mbox{Res}_{\ti{p}} \left[\frac{\left(\mathrm{tr}(J\tilde{\xi})\right)^2} {\tilde{\xi}_1\,\tilde{\xi}_2} d\tilde{z}_1 \wedge d\tilde{z}_2\right],
$$
where $\ti{\xi} = \tilde{\xi}_1 \dfrac{\partial}{\partial \ti{z}_1} + \tilde{\xi}_2 \dfrac{\partial}{\partial \ti{z}_2}$ is a vector field defining $\ti{\F}$ in a neighborhood of $\ti{p}$. As was done for the Poincar\'e-Hopf index, we define the Baum-Bott index of a foliation $\F$ on $\wP^2_w$ at an isolated singularity $p$ by
\begin{equation}\label{bbdef}
B\!B_{p}({\F}) =\frac{1}{\#G_p} \mbox{Res}_{\ti{p}} \left[\frac{\left(\mathrm{tr}(J\tilde{\xi})\right)^2} {\tilde{\xi}_1\,\tilde{\xi}_2} d\tilde{z}_1 \wedge d\tilde{z}_2\right]
\end{equation}
where $\ti{\xi} = \pi^\ast_p \xi$, for $\xi \in H^0(T\wP^2_w \otimes \scr{O}_{\wP^2_w}(d-1))$ defining  $\F$ and $\pi_p: (\wti{U}, \ti{p})\longrightarrow (U, p)$ a local smoothing covering, with $U$ a good representative of $\{p\}$. 
 
\begin{prop}\label{bb} Let $\xi$ be a holomorphic section of  $T\mathbb{P}_{w}^{2}\otimes \scr{O}_{{\wP^2_w}}(d-1)$ with isolated zeros only. Then
\begin{equation}\label{sumbb}
\sum_{p | \xi(p)=0} B\!B_{p}({\F}) =\frac{1}{w_{0}w_{1}w_{2}}(d +|w| -1)^2.
\end{equation}
\end{prop}

\begin{proof}

\noindent Since $\mathrm{tr}$ denotes the invariant symmetric polynomial $C_1$,  Theorem \ref{teo1} gives 
\begin{equation}\label{bb1}
\sum_{p | \xi(p)=0} B\!B_{p}({\F}) = \int_{\mathbb{P}^{2}_{w}} c^2_{1}(T\mathbb{P}^{2}_{w}
-\scr{O}_{\wP^2_w}(1-d))
\end{equation}
By (\ref{vir}) we have
\begin{eqnarray*}
\int_{\mathbb{P}^2_w}c_{1}^{2}(T\mathbb{P}_{w}^{2} - \scr{O}_{\wP^2_w}(1-d))
&=& \int_{\mathbb{P}^2_w}\big(c_{1}(T\mathbb{P}_{w}^{2})+c_{1}(\scr{O}_{\wP^2_w}(d-1))\big)^2\\
\mathrm{by \; (\ref{euler1})\quad} &=& \int_{\mathbb{P}^2_w}\left(|w| c_1 (\scr{O}_{\wP^2_w}(1))+(d-1)c_{1}(\scr{O}_{\wP^2_w}(1))\right)^2\\
&=& \frac{1}{\deg \varphi_w}\int_{\mathbb{P}^2}\varphi^\ast_w\big((d +|w| -1)^2 (c_{1}(\scr{O}_{\wP^2}(1)))^2\big)\\
&=& \frac{1}{w_{0}w_{1}w_{2}} (d + |w| -1)^2 \int_{\mathbb{P}^2} (c_{1}(\scr{O}_{\wP^2}(1)))^2\\
&=& \frac{1}{w_{0}w_{1}w_{2}} (d + |w| -1)^2.
\end{eqnarray*}
\end{proof}

Let $\F$ be a foliation on $\wP^2_w$ and suppose $p$ is an isolated singularity of $\F$. We say that $p$ is \emph{radial} provided the lifting $\ti{\xi} = \pi^\ast_p \xi$, of a vector field defining $\F$ around $p$ has radial linear part at $\ti{p}$, that is, $\ti{\xi} = \ti{z}_1\dfrac{\partial}{\partial \ti{z}_1} + \ti{z}_2\dfrac{\partial}{\partial \ti{z}_2} + \cdots$. Notice that, for a radial singularity, since $J(\ti{\xi})= I$ we have $\det J(\ti{\xi}) = 1$ and $\mathrm{tr}  J(\ti{\xi})=2$. It follows that 
\begin{equation}\label{ibb}
I_{{p}}({\xi}) = \dfrac{1}{\#G_p}\;\; \mathrm{and} \;\; B\!B_{{p}}({\F})=\dfrac{4}{\#G_p}.
\end{equation}

\begin{cor}\label{bbrad}
Let $\F$ be a holomorphic foliation of degree $d$ on $\wP^2_w$ with isolated singularities only.
If all the singularities of $\F$ are radial, then

\begin{equation}\label{ibbrad} 
d=\frac{1}{3}\left(3-C_1(w)\pm 2\sqrt{C_1(w)^{2}-3C_2(w)}\,\right).
\end{equation}
In particular, on $\wP(1,1,k)$ with $k \geq 1$, we have $d=\frac{k-1}{3}$ or $d=1-k$.

\end{cor}

\begin{proof} By Propositions \ref{enum1}, \ref{bb} and by (\ref{ibb}):
\begin{eqnarray*}
\frac{1}{w_{0}w_{1}w_{2}}(d+|w|-1)^2
&=& \sum_{p | \xi(p)=0} B\!B_{{p}}({\F})\\
&=& 4\, \sum_{p | \xi(p)=0} \mathcal{I}_{p}(\xi)\\
&=& 4\, \frac{(d-1)^2+ C_{1}(w) (d-1) +C_{2}(w)}{w_{0} w_{1} w_{2}}\\
\end{eqnarray*}
Noticing that $|w|=C_1(w)$ (\ref{ibbrad}) follows. 
\end{proof}

Let's comment on the case $\wP(1,1,k)$ in the above corollary. If $k=1$ then $d=0$ and, up to automorphism, we have only the foliation on $\wP^2$ given by $\omega= z_1 dz_0 - z_0 dz_1$, which is a pencil. For $k>1$, up to automorphism, we also have only one foliation of degree $1-k$, analogously the pencil defined by $\omega= z_1 dz_0 - z_0 dz_1$. Hence, if $k>1$ and $3 \nmid (k-1)$ then radial foliations are pencils.

\subsection{Hirzebruch surfaces as good resolutions of $\wP(1,1,k)$}

The $k$-th Hirzebruch surface $\Sigma_k$, $k \geq 0$, is defined as $\wP(\scr{O}_{\wP^1} \oplus \scr{O}_{\wP^1}(k))$. These are $\wP^1$ bundles over $\wP^1$, hence rational ruled surfaces and they have just one ruling, except for $k=0$ since $\Sigma_0 = \wP^1 \times \wP^1$. J. Matsuzawa \cite{ichi} gave the following convenient realization of $\Sigma_k$:
\begin{equation}\label{H1}
\Sigma_k = \{([z_0:z_1:z_2], [s:t]) \in \wP^2 \times \wP^1 : s^k z_0 = t^k z_1\}.
\end{equation}
The projection onto the second factor $\mathrm{pr_2}: \Sigma_n \longrightarrow \wP^1$ gives the ruling and, for $k \geq 1$,
\begin{equation}\label{H2}
D=\{([z_0:z_1:z_2], [s:t]) \in \Sigma_k : z_0 = z_1 =0\}
\end{equation}
is the unique section with $D\cdot D =-k$. Also, $\mathrm{Pic}(\Sigma_k) = \Z D \oplus \Z L$ where $L$ is a fiber, 
$D \cdot L = 1$ and $L \cdot L =0$.

Consider the orbifold $\wP(1,1,k)$, $k \geq 1$, which we denote by $\wP^2_k$ and recall that, for $k>1$, $Sing(\wP^2_k) = \{\mathsf{e}_2 = [0: 0: 1]_k \}$. Let
\begin{equation}\label{res1}
\mathbb{V}_k = \{([z_0:z_1:z_2]_k, [s:t])\in \mathbb{P}^{2}_k \times \mathbb{P}^{1} \,:\, s z_{0}= t z_{1}\}.
\end{equation}
In view of (\ref{H1}), the map
\begin{equation}\label{res2}
\begin{array}{ccc}
\wP^2 \times \wP^1 &\umapright{\displaystyle\Phi}& \wP^2_k \times \wP^1 \\
([z_0:z_1:z_2], [s:t]) & \longmapsto & ([z_0^{1/k}:z_1^{1/k}:z_2]_k, [s:t])
\end{array}
\end{equation}
satisfies
\begin{equation}
\Phi_{|\Sigma_k \setminus D} : \Sigma_k \setminus D \longrightarrow \mathbb{V}_k \setminus \{\mathsf{e}_2\}
\end{equation}
is an isomorphism, in the sense of Definition \ref{v-m}.
The composite morphism 
\begin{equation}
\sigma = \pi_1 \circ \Phi_{|\Sigma_k} : 
\Sigma_k \longrightarrow \wP^2_k
\end{equation}
where $\pi_1 : \wP^2_k \times \wP^1 \longrightarrow \wP^2_k$ is the projection onto the first factor, gives a good resolution of $\wP^2_k$ with exceptional divisor $D$.
\begin{prop}\label{Hres}
Let $\xi \in H^0(\mathbb{P}^2_k, T\mathbb{P}^2_k\otimes \scr{O}_{\wP^2_k}(d-1) )$ be a section with isolated zeros only. Write $\sigma^\ast \xi = \tilde{\xi}$. Then
\begin{equation}\label{Hloc}
\sum\limits_{\tilde{p}\, \in D | \tilde{\xi}=0} {I}_{\tilde{p}}(\tilde{\xi})
- {I}_{\mathsf{e}_2}(\xi)= (d^2+kd+k)\left(k-\frac{1}{k}\right).
\end{equation}
\end{prop}
\begin{proof}
From (\ref{teo2b}) of Theorem \ref{teo2} we have
\begin{equation}\label{Hloc1}
I_{\mathsf{e}_2}(\xi)= \sum\limits_{\tilde{p}\, \in D | \tilde{\xi}=0} {I}_{\tilde{p}}(\tilde{\xi}) - 
\int\limits_{\Sigma_k} c_2 \left(\mathsf{e}_2, (\sigma^\ast(T\wP^2_k \otimes \scr{O}_{\wP^2_k}(d-1)))\right)
\end{equation}
and hence we must calculate $\displaystyle\int_{\Sigma_k} c_2 \left(\mathsf{e}_2, (\sigma^\ast(T\wP^2_k \otimes \scr{O}_{\wP^2_k}(d-1)))\right)$.
From (\ref{teo2a}) and Proposition (\ref{enum1}) we have, recalling that $w=(1,1,k)$,
\begin{equation}\label{Hloc2}
\begin{array}{ccc}
\displaystyle\int\limits_{\Sigma_k} c_2 \left(\mathsf{e}_2, (\sigma^\ast(T\wP^2_k \otimes \scr{O}_{\wP^2_k}(d-1)))\right) &{}&\, \\
=\displaystyle\int\limits_{\Sigma_k} c_2 \left(\sigma^\ast(T\wP^2_k \otimes \scr{O}_{\wP^2_k}(d-1) )\right) &-& \dfrac{(d-1)^2+(d-1)C_{1}(w)+C_{2}(w)}{k} \hfill\\
=\displaystyle\int\limits_{\Sigma_k} c_2 \left(\sigma^\ast(T\wP^2_k \otimes \scr{O}_{\wP^2_k}(d-1) )\right) &-& \dfrac{(d-1)^2+(d-1)(k+2)+(2k+1)}{k} \hfill\\
=\displaystyle\int\limits_{\Sigma_k} c_2 \left(\sigma^\ast(T\wP^2_k \otimes \scr{O}_{\wP^2_k}(d-1) )\right) &-& \dfrac{d^2 +kd+k}{k}.\hfill
\end{array}
\end{equation}
Now,
$$\begin{array}{ccc}
c_{2}(\sigma^{\ast}(T\mathbb{P}^2_k \otimes \scr{O}_{\wP^2_k}(d-1)))
&=& c_{2}(\sigma^{\ast} T\mathbb{P}^2_k) + c_{1}(\sigma^{\ast} T\mathbb{P}^2_k ) c_{1}(\sigma^{\ast}\scr{O}_{\wP^2_k}(d-1))\nonumber\\
&+ &  \left(c_{1}(\sigma^{\ast}\scr{O}_{\wP^2_k}(d-1))\right)^2\hfill
\end{array}
$$
and by using Euler's sequence (\ref{euler}) we get $c_{2}(\sigma^{\ast} T\mathbb{P}^2_k) = (2k+1) \left(c_{1}(\sigma^{\ast}\scr{O}_{\wP^2_k}(1))\right)^2$ and
$c_{1}(\sigma^{\ast} T\mathbb{P}^2_k) = (k+2) c_{1}(\sigma^{\ast}\scr{O}_{\wP^2_k}(1))$, from which we deduce
\begin{equation}\label{Hloc3}
c_{2}(\sigma^{\ast}(T\mathbb{P}^2_k \otimes \scr{O}_{\wP^2_k}(d-1))) = (d^2+kd+k) \left(c_{1}(\sigma^{\ast}\scr{O}_{\wP^2_k}(1))\right)^2.
\end{equation}
By considering the divisor $S=\{z_0=0\}$ in $\mathbb{P}^2_k$ and noticing that $\scr{O}_{\wP^2_k}(1) \cong \scr{O}_{\wP^2_k}(S)$, a straigthforward calculation gives $c_1(\sigma^{\ast}\scr{O}_{\wP^2_k}(1)) = kL+D$, $L$ a fiber, $L \cdot D=1$, $L^2=0$ and hence $\displaystyle\int_{\Sigma_k}\left(c_{1}(\sigma^{\ast}\scr{O}_{\wP^2_k}(1))\right)^2=k$. Substituting into (\ref{Hloc2}) we obtain the result.
\end{proof}

The next result gives a condition for a foliation to have a singularity at $Sing(\wP^2_k)$. We will give two proofs of it, one as a consequence of Proposition \ref{enum1} and the other by using Proposition \ref{Hres}.
\begin{prop}
Let $\F$ be a foliation of degree $d$ on $\wP^2_k$, $k>1$. Suppose $\F$ has only isolated singularities. If $k \nmid \,d^2$ then $\F$ has a singularity at $Sing(\wP^2_k) = \{\mathsf{e}_2\}$. 
\end{prop} 

\noindent\emph{First proof.} In fact, suppose $\mathsf{e}_2$ is non-singular for $\F$. Then the right hand side of (\ref{enum2}) is an integer and this happens only if $k \mid d^2$.\hfill $\square$\\
\noindent\emph{Second proof.} If $\mathcal{F}$ is non-singular at  $\mathsf{e}_2$, then (\ref{Hloc}) reads
$$\sum\limits_{\tilde{p}\, \in D | \tilde{\xi}=0} {I}_{\tilde{p}}(\tilde{\xi})= (d^2+kd+k)\left(k-\frac{1}{k}\right).$$
Since the left hand side is an integer and $k>1$ we must have $k \mid d^2$.
\hfill $\square$

\begin{obs} \rm{Still in $\wP^2_k$ we observe that, in case $d \geq 0$ and $\mathsf{e}_2 \in Sing(\F)$ is non-degenerate, then $Sing(\F)$ contains at least another point. This is because, in case this does not happens, then (\ref{enum2}) reads $ \dfrac{1}{k}=\dfrac{d^2+kd+k}{k}$, which is impossible since $k>1$ and $d\geq 0$.}
\end{obs}

\begin{ex}\label{nonsing}
\rm{Let $\eta = (k+1) f dg - g df$ where $f(z_0,z_1,z_2) = z_0^{k+1} + z_1^{k+1} - (z_0 + z_1) z_2$ and $g(z_0, z_1, z_2)= az_0^k + b z_1^k +c z_2$, $k>1$. $\eta$ induces a foliation $\mathcal{F}$ of degree $k$ on $\wP^2_k$, with isolated singularities and $\mathcal{F}$ is non-singular at $\sf{e}_2$.}
\end{ex}

\noindent{\footnotesize
\textsc{Acknowlegments.} The third named author is grateful to IMPA and to the Univ. of Valladolid for hospitality.




\end{document}